\documentclass[12pt]{article}
\usepackage{latexsym,amsfonts,amssymb,mathrsfs,float}
\usepackage{dsfont}
\usepackage{amsthm}
\usepackage[numbers]{natbib}
\usepackage{amsmath}
\usepackage{indentfirst}
\usepackage{bbm}

\usepackage[dvips]{color}
\usepackage{colordvi,multicol}
\allowdisplaybreaks

\def \cal{\mathcal}

\newtheorem{thm}{Theorem}[section]

\newtheorem{lem}[thm]{Lemma}
\newtheorem{pro}[thm]{Proposition}

\newcommand{\wid}{\widetilde}

\numberwithin{equation}{section}

\begin{document}

\title{\bf Favorite sites of one-dimensional asymmetric simple random walk}
 \author{Guangshuo Zhou$^1$, Zechun Hu$^{1,}$\footnote{Corresponding author}\ , and  Renming Song$^{2}$\\ \\
  {\small $^1$ College of Mathematics, Sichuan  University,
 Chengdu 610065, China}\\
 {\small zhgsh3363@163.com, zchu@scu.edu.cn}\\ \\
  {\small $^2$Department of Mathematics,
University of Illinois Urbana-Champaign, Urbana, IL 61801, USA}\\
 {\small rsong@illinois.edu}}

 \date{}
\maketitle

 \noindent{\bf Abstract:}\quad
 In this paper, we study  favorite sites of one-dimensional asymmetric simple random walks. We show that almost surely,  for any fixed integer $r\geq 1$,  ``$r$ favorite sites" occurs infinitely often.
 We also give the asymptotic growth rate of the number of favorite sites.

\smallskip


\smallskip


\section{Introduction and main results}\setcounter{equation}{0}

Let $X_i,i=1,2,\ldots,$ be independent and identically distributed random variables with
$$
\mathbf{P}(X_1 =1)=p~~ \text{and}~~ \mathbf{P}(X_1=-1)=1-p=:q.
$$
For any $x\in  \mathbb{Z}$, define $S_0=x$, $S_n = x+X_1 + X_2 +\cdots+X_n$ for $n\ge 1$. Then
$\{S_n\}_{n\geq 0}$ is a one-dimensional simple random walk starting from $x$. We will use $\mathbf{P}_x$ to denote the law of this random walk and $\mathbf{E}_x$ to denote the corresponding expectation. We will write $\mathbf{P}_0$ and $\mathbf{E}_0$ as $\mathbf{P}$ and $\mathbf{E}$, respectively. In this paper, we are concerned with
one-dimensional asymmetric simple random walks and so we will assume $p\neq q$. Without loss of generality, we will assume $p>q$.

For any $z\in \mathbb{Z}, n\in \mathbb{N}, n\geq 1$, define
\begin{equation*}\label{eq:1.1}
  \xi(z,n) := \# \{k: 0<k \leq n,~S_k =z \},
\end{equation*}
where $\# D$ denotes the cardinality of the set $D$. $\xi(z, n)$ is called the local time of the random walk at $z$ up to time $n$. Define $\xi(n)=\max_{y\in \mathbb{Z}}\xi(y,n)$. A site $z\in \mathbb{Z}$ is called a favorite site at time $n$ if $\xi(z, n)=\xi(n)$.  The set
$$
 \mathcal{K}(n):= \left\{  z\in\mathbb{Z}: \xi(z,n)=\xi(n)\right\}
$$
is the set of all the favorite sites of the random walk at time $n$.
In this paper we study the local time and favorite sites of the one-dimensional
asymmetric random walk.

The study of the local time of transient random walks was initiated by the papers Dvoretzky-Erd\H{o}s \cite{DE1951} and Erd\H{o}s-Taylor \cite{ET1960}. These two papers examined properties of the symmetric simple random walk in dimensions $d\geq 3$.
Cs\'{a}ki et al. \cite{CFR2008} proved several  properties of the local time of the one-dimensional asymmetric simple random walk and showed that its behavior is similar to that of the simple symmetric random walks in
dimensions $d\geq 3$.

The study of the favorite site set $\mathcal{K}(n)$ started with the papers Bass-Griffin \cite{BG1985} and Erd\H{o}s-R\'{e}v\'{e}sz \cite{ER1984} for symmetric simple random walks. In \cite{ER1984} and \cite{ER1987}, Erd\H{o}s and R\'{e}v\'{e}sz posed the question whether $\#\mathcal{K}(n)=r$ occurs infinitely often (i.o. for short) for $r\geq 3$ for $d$-dimensional symmetric random walks. For $d\ge 3$, Erd\H{o}s-R\'{e}v\'{e}sz  \cite{ER1991} showed that, almost surely,  $\#\mathcal{K}(n)=r$ i.o. for any integer $r \geq 1$. For one-dimensional symmetric random walks, the related question of favorite edges was first studied. T\'{o}th-Werner \cite{TW1997} proved that almost surely there are only finitely many times at which there are four distinct favorite edges.  Then T\'{o}th \cite{T2001} studied the favorite site problem for one-dimensional symmetric random walks and proved that eventually  $\# \mathcal{K}(n) \leq 3$ almost surely. However, the question whether three favorite sites  occur infinitely often almost surely was open for nearly 20 years until Ding-Shen \cite{DS2018} settled this affirmatively. Similar phenomena for favorite edges were obtained by Hao et al. \cite{HHMS2024}. Recently, Hao et al. \cite{HLOZ2024} proved that $\#\mathcal{K}(n) =3$  i.o. for  two-dimensional symmetric  simple random walks  and  derived  sharp asymptotics for  $\#\mathcal{K}(n)$ for $d$-dimensional symmetric  simple random walks, $d\ge 3$.

The main goal of this paper is to prove the following results for one-dimensional asymmetric simple random walks.

\begin{thm}\label{thm:1.1}
  For any given integer $r\geq 1$,  ``$r$ favorite sites" occurs infinitely often almost surely.
\end{thm}

\begin{thm}\label{thm:1.2}
  Almost surely,
  \begin{equation}\label{thm1.2-a}
    \limsup_{n \to \infty} \frac{ \# \mathcal{K}(n) }{ \log \log n} = -\frac{1}{ \log ( 1-2q)}.
  \end{equation}
\end{thm}

We end this section with some
notations
that will be used later in this paper:
\begin{align}\label{gamma-h-lambda}
\gamma:=1-2q, \quad h:=\frac{q}{p}, \quad \lambda:=-\frac1{\log(2q)}
\end{align}
and
\begin{align}\label{theta-delta}
\theta:=-\frac1{\log\gamma}, \quad
\delta:=\frac{2 \log \frac{2q+h^{1/2}}{1+h^{1/2}}}{\log(2q)} -1.
\end{align}
It is elementary to check that $q<\frac12$ is equivalent to $\delta>0$.

For any real numbers $a$ and $b$, set $a\wedge b:= {\rm min}\{a,b\}$ and $a \vee b :={\rm max}\{a,b\}$. For any set $A$,  we use $1_A$ to denote the characteristic function of $A$.
 For any $r \in \mathbb{R}$, $[r]$ stands for the largest integer less than or equal to $r$.
 For integers $0\leq i <j$, set $S_{(i,j)}= (S_{i+1}, \ldots, S_{j-1})$ and $S_{[i,j]}= (S_i, \ldots, S_j)$. We define $S_{[i,j)}$ similarly.

 The rest of the paper is organized as follows. In Section 2, we collect some
preliminary results. The proofs of Theorems \ref{thm:1.1} and  \ref{thm:1.2} are given in Sections 3.

\section{Preliminaries}\setcounter{equation}{0}

We first recall some well-known facts about one-dimensional asymmetric simple random walks. The following assertion about the probability of no return can be found on \cite[p. 274]{Feller1968}, the assertion on the probability $\gamma(n)$ of no returns to the starting point in the $n-1$ steps is
given in \cite[(4.1)]{CFR2008}, \eqref{eq:2.1} is given in \cite[(4.8)]{CFR2008}.

\begin{lem}\label{lem:2.1}
It holds that
  \begin{equation*}
    \mathbf{P}(S_i \neq 0, i=1,2,...) = \gamma
  \end{equation*}
and that
\begin{equation*}
  1=\gamma(1) \geq \gamma(2) \geq \cdots \geq \gamma(n) \geq
\cdots \geq \gamma.
\end{equation*}
Furthermore, it also holds that
\begin{equation}\label{eq:2.1}
  \gamma(n) -\gamma  = O\left( \frac{(4pq)^{n/2}}{n^{3/2}} \right) .
\end{equation}
\end{lem}

The following elementary result can be found on \cite[p. 68]{Feller1968}.

\begin{lem}\label{lem:2.2}
For $z\in \mathbb{Z}$, if $n+z$ is an even number, then
\begin{equation*}
  \mathbf{P}(S_n=z) =\binom{n}{\frac{n+z}{2}} p^{\frac{n+z}{2}} q^{\frac{n-z}{2}}.
\end{equation*}
\end{lem}

It follows from Lemma \ref{lem:2.2} that, for any  positive integer $m$,
\begin{align}\label{eq:2.2}
  \mathbf{P}\left(S_0 \in S_{[m / 2, \infty) }\right)&= \mathbf{P}\left(\bigcup_{n \geq \frac{m}{2}} \{ S_n =0 \} \right) \leq \sum_{n\geq \frac{m}{2}} \binom{n}{[\frac{n}{2}]} (pq)^{\frac{n}{2}}\nonumber\\
  & \leq \sum_{n\geq \frac{m}{2}} 2^n (pq)^{\frac{n}{2}}\leq \frac{(4pq)^{m/4}}{1-\sqrt{4pq}}.
\end{align}

For $z\in \mathbb{Z}$, set
$$
\xi(z,\infty):=\#\{k>0: S_k=z\}.
$$
The following result can be found on \cite[p. 1048]{D1967AMS}. Recall that $h=\frac{q}{p}$.

\begin{lem}\label{lem:2.3}
  For $z\in\mathbb{Z}$ and $k>0$, we have
  \begin{eqnarray*}
 \mathbf{P}(\xi(z,\infty)=k)=\left\{
\begin{array}{cl}
h^{-z}(2q)^{k-1} (1-2q), &\mbox{if}\ z<0, \\
~~~~(2q)^k ~~(1-2q),&\mbox{if}\ z=0, \\
~~~~(2q)^{k-1} (1-2q),&\mbox{if}\ z>0.
\end{array}
\right.
\end{eqnarray*}
\end{lem}

The following result is \cite[Theorem 3.1]{CFR2008}.

\begin{lem}\label{lem:2.4}
  It holds that
 \begin{equation}\label{2.3}
    \lim_{n \to \infty} \frac{\xi(n)}{\log n} =\lambda.
  \end{equation}
\end{lem}

\bigskip

\begin{pro}\label{pro:2.5}
For any  $A>0$ and $n\in \mathbb{Z}^+:=\{z\in \mathbb{Z}: z>0\}$, we have
  \begin{equation*}
  \sup_{z\in \mathbb{Z}^+}
    \mathbf{P}(\xi(0,\infty)+\xi(z,\infty)> 2 \lambda A \log n) \leq C n^{-A(1+\delta)}.
  \end{equation*}
\end{pro}

\noindent {\bf Proof.} It follows from \cite[Lemma 5.1]{CFR2008} that there exists a constant $C$ such that for all sufficiently large $u$,
\begin{equation*}
    \mathbf{P}(\xi(0,\infty)+\xi(z,\infty) \geq u) \leq
        C \bigg( \frac{2q+h^{z/2}}{1+h^{z/2}}\bigg)^u\leq C \bigg( \frac{2q+h^{1/2}}{1+h^{1/2}}\bigg)^u.
\end{equation*}
Hence
  \begin{align*}
 &\sup_{z\in \mathbb{Z}^+} \mathbf{P}(\xi(0,\infty)+\xi(z,\infty)\geq 2 \lambda A \log n)\\
  \leq &  C \exp\left[\log\left(\frac{2q+h^{1/2}}{1+h^{1/2}}\right)
  \cdot 2 \lambda A \log n\right]\\
  \leq & C \exp\left(-\frac{1}{2} (\delta+1) \lambda^{-1} \cdot 2 \lambda A \log n\right)
  =    C n^{-A(\delta +1)}.
    \end{align*}
\hfill\fbox

We now introduce some stopping times and related events that will used frequently
in the next section.
For any integers $m \ge 1$ and $k\ge 0$, we define the stopping times $T_m^k$ and corresponding locations $L_m^k$ by
\begin{equation*}\label{eq:2.3}
T_m^0:=0, ~~T_m^k:=\inf \left\{n>T_m^{k-1}: \#\left\{z \in \mathbb{Z}: \xi(z, n) \geq m\right\}=k\right\} \text { for } k \geq 1 \text { and } L_m^k:=S_{T_m^k}.
\end{equation*}
$T_m^k$ is the first time that the random walk has visited $k$ distinct sites,  each at least $m$ times, and $L_m^k$ is the location of the $k$-th such site.
Define $\mathcal{F}_m^k:=\sigma\left\{S_{\left[0, T_m^k\right]}\right\}$.

For any positive integers $m, k$, we define
\begin{align*}
A_m^k&:=\left\{S_n \notin\left\{L_m^1, \ldots, L_m^{k-2}\right\}, \text { for any } n \in\left[T_m^{k-1}+m / 2, T_m^k \wedge T_{m+1}^1\right]\right\} \\
&\quad\quad \cap\left\{S_n \neq L_m^{k-1}, \text { for any } n \in\left(T_m^{k-1}, T_m^k \wedge T_{m+1}^1\right]\right\},
\end{align*}
which is the event that $\left\{S_n\right\}_{n\geq0 }$ does not hit the sites $\left\{L_m^1, \ldots L_m^{k-2}\right\}$ between $T_m^{k-1}+\frac{m}{2}$ and $T_m^k \wedge T_{m+1}^1$ and also does not hit $L_m^{k-1}$ between $T_m^{k-1}+1$ and $T_m^k \wedge T_{m+1}^1$.
We also define
\begin{equation*}
\widetilde{A}_m^k:=\left\{S_n \notin\left\{L_m^1, \dots, L_m^{k-2}\right\}, \text { for any } n \in\left[T_m^{k-1},\left(T_m^{k-1}+m / 2\right) \wedge T_m^k \wedge T_{m+1}^1\right)\right\}
\end{equation*}
and
\begin{align}\label{eq:2.5}
B_m^k:=A_m^1 \cap \cdots \cap A_m^k ; \quad \widetilde{B}_m^k:=\widetilde{A}_m^2 \cap \cdots \cap \widetilde{A}_m^k.
\end{align}

Let
\begin{align}\label{eq:2.6}
C_m^k:=\left\{\exists~ n \geq 0 \text { s.t. } \# \mathcal{K}(n)=k \text { and } \xi(n)=m\right\}=\left\{T_m^k<T_{m+1}^1\right\}.
\end{align}
Since $\xi(n) \rightarrow \infty$ almost surely, for any $k \in \mathbb{Z}^{+}$, we have
\begin{equation*}
\left\{\# \mathcal{K}(n) \geq k \text { infinitely often in } n\right\} \stackrel{\text { a.s. }}{=}\left\{C_m^k \text { infinitely often in } m\right\}.
\end{equation*}
Since $C_m^k=\left\{T_m^k<T_{m+1}^1\right\}$, we know that
$C_m^k \in \mathcal{F}_{m+1}^1$. Note that,  by  the definitions above,
\begin{equation}\label{eq:2.7}
C_m^k=B_m^k \cap \widetilde{B}_m^k.
\end{equation}

\section{Proofs of Theorems \ref{thm:1.1} and  \ref{thm:1.2}}\setcounter{equation}{0}

To prepare for  the proof Theorem \ref{thm:1.1}, we first make some preparations.

\begin{lem}\label{lem3.3}
Suppose that for any $j=1,2,\ldots,n$, $\mathbf{P}\left(A_j\right)=c>0$. Then for any
$k=1,2,\ldots,n-1$,
we have
\begin{equation*}
\mathbf{P}\left(\sum_{j=1}^n 1_{A_j}>k\right) \geq c-\frac{k(k+1)}{2n} .
\end{equation*}
\end{lem}

\noindent {\bf Proof.} For any $m=0,1,\ldots,n$, let $Q_m:=\left\{\sum_{j=1}^n 1_{A_j}=m\right\}$. Then
\begin{equation*}
\sum_{m=0}^n m  \mathbf{P}\left(Q_m\right)=\mathbf{E}\left(\sum_{j=1}^n 1_{A_j}\right)=nc,
\end{equation*}
which implies that
\begin{align*}
\mathbf{P}\left(\sum_{j=1}^n 1_{A_j}>k\right) & =\sum_{m=k+1}^n \mathbf{P}\left(Q_m\right) \geq \sum_{m=k+1}^n \frac{m}{n} \mathbf{P}\left(Q_m\right) \\
&= c-\frac{1}{n} \sum_{m=0}^k m \mathbf{P}\left(Q_m\right)\\
& \geq c-\frac{1}{n} \sum_{m=0}^k m=c-\frac{k(k+1)}{2n} .
\end{align*}
\hfill\fbox

\begin{lem}\label{lem:3.2}
Let $\left\{A_n\right\}_{n \geq 1}$ be a sequence of events and $c \in(0,1]$. If $\mathbf{P}\left(A_n\right) \geq c$ for all $n\ge 1$,  then
$
\mathbf{P}\left(A_n~ \text{\rm i.o.} \right) \geq c .
$
\end{lem}

\noindent {\bf Proof.}
By Lemma \ref{lem3.3}, we know that for any $k \geq 1$,
\begin{equation*}
\mathbf{P}\left(\sum_{j=1}^{\infty} 1_{A_j}>k\right) \geq c.
\end{equation*}
It follows that
\begin{equation*}
\mathbf{P}\left(\sum_{j=1}^{\infty} 1_{A_j}=\infty\right)=\lim _{k \rightarrow \infty} \mathbf{P}\left(\sum_{j=1}^{\infty} 1_{A_j}>k\right) \geq c.
\end{equation*}
\hfill\fbox

\bigskip

Now we are ready to prove Theorem \ref{thm:1.1}.

\noindent {\bf Proof of Theorem \ref{thm:1.1}.}
For any positive integers $k, m$ and $j=1,2,\ldots,k$, let
$$
\wid{C}_m^j := \{ S_n > S_{T_m^{j-1}}, ~\text{for any}~ n\in(T_m^{j-1}, T_m^j] \}.
$$
By the strong Markov property and Lemma \ref{lem:2.1}, we have
\begin{equation*}
\begin{aligned}
\mathbf{P}\left(\wid{C}_m^{1} \cap \wid{C}_m^{2}\right)
&= \mathbf{P}\left(S_n>0, \forall n \in\left(0, T_m^{1}\right]; S_n>S_{T_m^{1}}, \forall n \in\left(T_m^{1}, T_m^{2}\right]\right) \\
&=\mathbf{E}\left(1_{\{S_n>0, \forall n \in\left(0, T_m^{1}\right]\}}
\cdot \mathbf{P}_{L_{T_m^1}}(S_n>S_0,\forall n\in(0,T_m^2-T_m^1])\right)\\
&\geq \mathbf{E}\left(1_{\{S_n>0, \forall n>0\}}
\cdot \mathbf{P}_{L_{T_m^1}}(S_n>S_0,\forall n>0)\right)\\
& =(1-2q)^2 .
\end{aligned}
\end{equation*}
Using induction, we easily get
\begin{equation*}
\mathbf{P}\left(\wid{C}_m^1 \cap \cdots \cap \wid{C}_m^k\right) \geq (1-2q)^k.
\end{equation*}
Since $\wid{C}_m^1\cap \cdots \cap \wid{C}_m^k\subseteq C_m^k$, we have
\begin{align}\label{e:step1}
\mathbf{P}\left(C_m^{k}\right) \geq (1-2q)^k.
\end{align}

Combining \eqref{e:step1} with Lemma \ref{lem:3.2},
we immediately get that
\begin{align}\label{e:step2}
\mathbf{P}(``k \mbox{ favorite sites" occurs i.o.})\ge (1-2q)^k.
\end{align}

Using the transience of  $\left\{S_n\right\}_{n \geq 0}$, one can easily see that the  favorite sites process $(\mathcal{K}(n))_{n\geq 1}$ is transient, i.e., a fixed site cannot be a favorite site infinitely often.

Let $g(k):=\#\{n \geq 1: \# \mathcal{K}(n)=k\}$.
By the proof in \cite[Section 3.3]{DS2018} and the transience of $(\mathcal{K}(n))_{n\geq 1}$, we get hat $\{g(k)=\infty\}$ is a tail event.
Then by Kolmogorov's 0-1 law and
\eqref{e:step2}, we get that
\begin{align*}\label{eq:3.1}
\mathbf{P}(g(k)=\infty)=1.
\end{align*}
\hfill\fbox

Now we turn to the proof of Theorem \ref{thm:1.2}. We first prove two lemmas.

For $\epsilon \in(0,1)$, define
\begin{equation*}
D_n^\epsilon:=\left\{\left(1-\frac{\epsilon}{2}\right) \lambda \log n \leq \xi(n) \leq 2\left(1+\frac{\epsilon}{2}\right) \lambda \log n\right\},
\end{equation*}
where $\lambda$ is defined in \eqref{gamma-h-lambda}. Let
\begin{equation*}
E_n^\epsilon:=\left\{\mathrm{F}_n=0\right\},
\end{equation*}
where
\begin{equation*}
\mathrm{F}_n:=\sum_{i=1}^n \sum_{j=i+1}^{(i+2 \lambda \log n) \wedge n} 1_{\left\{\min \left(\xi\left(S_i, n\right), \xi\left(S_j, n\right)\right) \geq(1-\epsilon) \lambda \log n, S_i, S_j \notin S_{(i, j)}, S_i \neq S_j\right\}}.
\end{equation*}
In other words, $\mathrm{E}_n^\epsilon$ is the event that there exist no pair of thick points (\text{i.e.} sites with local time at least $(1-\epsilon) \lambda \log n$) that lie ``close together" before time $n$.

Recall that $\delta$ is defined in \eqref{theta-delta}. Fix $\rho>0$ such that $\rho \delta / 2>1+\delta$.

\begin{lem}\label{lem:4.1}
For any $\epsilon \in(0,1)$ with $(1+\delta)(1-\epsilon)^2 >1+\delta /2$ , there exists $C=C(\epsilon)>0$ such that for $N$ large,
\begin{equation*}\label{eq:3.4}
\mathbf{P}\left(\bigcup_{n \geq N}\left(D_n^\epsilon \cap E_n^\epsilon\right)^{\mathrm{c}}\right) \leq C N^{-(\epsilon\wedge( \delta/\rho))}.
\end{equation*}
Consequently,
\begin{equation*}
\mathbf{P}\left(D_n^\epsilon \cap E_n^\epsilon \text { occurs eventually }\right)=1.
\end{equation*}
\end{lem}

\begin{proof}[\bf Proof.] By Lemma \ref{lem:2.3} and the
definition of $\lambda$ in \eqref{gamma-h-lambda},
we have for $z\in \mathbb{Z}$
\begin{align*}
\mathbf{P}\left(\xi(z,n) \geq 2(1+\frac{\epsilon}{2}) \lambda \log n\right)
& \leq \sum_{k \geq 2(1+\epsilon/2) \lambda \log n} \mathbf{P} (\xi(z,n)=k)\\
&\leq \sum_{k \geq 2(1+\epsilon/2) \lambda \log n } \mathbf{P}(\xi(z, \infty
)=k) \\
&\leq \sum_{k \geq 2(1+\epsilon/2) \lambda \log n } (2q)^{k-1} (1-2q)\\
&\leq (1-2q)\cdot \frac{(2q)^{2(1+\epsilon/2) \lambda \log n-1}}{1-2q}\\
&\leq C n^{-2(1+\epsilon/2)}.
\end{align*}
It follows that when $n$ is large enough so that $2(1+\frac{\epsilon}{2}) \lambda \log n>1$,
\begin{align}\label{eq:4.1}
\mathbf{P}\left(\xi(n)\geq 2(1+\frac{\epsilon}{2}) \lambda \log n\right)
&=\mathbf{P}\left(\bigcup_{k=-(n-1)}^{n-1}\left\{\xi(k,n)\geq 2(1+\frac{\epsilon}{2}) \lambda \log n\right\}\right)\notag\\
&\leq \sum_{k=-(n-1)}^{n-1}\mathbf{P}\left(\xi(k,n)\geq 2(1+\frac{\epsilon}{2}) \lambda \log n\right)\notag\\
& \leq 2n \cdot C n^{-2(1+\epsilon/2)}=C n^{-(1+\epsilon)}.
  \end{align}

Let
\begin{equation*}
u:=\left[(1-\frac{\epsilon}{2})\lambda\log n\right],\quad v:=\left[\frac{n}{u^2}\right].
\end{equation*}
Recall that $\gamma(u+1)$ is the probability that the random walk does not return to the starting point in the first $u$ steps. According to \eqref{eq:2.1}, we have
\begin{equation*}
  1 - \gamma(u+1) = 1 -\gamma- O\left(\frac{(4pq)^{(u+1)/2}}{(u+1)^{3/2}}\right)=2q+o(1),~~ u\to\infty.
\end{equation*}
For $k\ge 1$, let $T_k$ be the $k$-th time that the random walk returns to the starting point.
By the strong Markov property, we have that
$$
\mathbf{P}(T_1\leq u, T_k-T_{k-1}\leq u, k=1,2,\ldots,u)
=
\left( 1-
\gamma(u+1)
\right)^{[(1-\epsilon/2) \lambda \log n]}=:\beta.
$$
Since there are at least $v$ such independent segments in the first $n$ steps,
we have
\begin{equation}\label{eq:4.2a}
  \mathbf{P}\left(\xi(n) \leq (1-\frac{\epsilon}{2}) \lambda \log n\right)\leq (1-\beta)^v.
\end{equation}
By the definition \eqref{gamma-h-lambda} of $\lambda$, we have
\begin{equation*}
  \log\beta = \left[(1-\frac{\epsilon}{2}) \lambda \log n \right] \left(\log(2q)+o(1)\right)
=-\left(1-\frac{\epsilon}{2}\right)\log n + o(\log n),
\end{equation*}
so
$$
\beta = n^{-(1-\epsilon/2)+o(1)}.
$$
Consequently,
$$
\beta v =
n^{-(1-\epsilon/2)+o(1)}\cdot\left[\frac{n}{[(1-\frac{\epsilon}2)\lambda \log n]^2}\right].
$$
Therefore for any $\alpha\in(0,\epsilon/2)$, we have $\beta v\geq n^\alpha$ for all $n$ large, and hence
for all $n$ large,
\begin{equation*}
(1-\beta)^v \leq e^{-\beta v} \leq e^{-n^\alpha}.
\end{equation*}
Hence by \eqref{eq:4.2a}
we obtain that
for all $n$ large,
\begin{equation}\label{eq:4.2}
\mathbf{P}\left(\xi(n)\leq (1-\frac{\epsilon}{2})\lambda\log n\right)
\leq (1-\beta)^v \leq e^{-\beta v}\leq e^{-n^\alpha}.
\end{equation}
Combining \eqref{eq:4.1}, \eqref{eq:4.2} and the definition of  $D_n^\epsilon$, we get
\begin{equation}\label{eq:4.3}
\begin{aligned}
  \mathbf{P}((D_n^\epsilon)^c) & =\mathbf{P} \left( \xi(n) \leq (1-\frac{\epsilon}{2}) \lambda \log n\ \mbox{or}\ \xi(n) \geq 2(1+\frac{\epsilon}{2}) \lambda \log n  \right)  \\
  & \leq \mathbf{P} \left( \xi(n) \leq (1-\frac{\epsilon}{2}) \lambda \log n \right)+ \mathbf{P}\left( \xi(n) \geq 2(1+\frac{\epsilon}{2}) \lambda \log n  \right)  \\
   &\leq e^{-n^{\alpha}} + C n^{-(1+\epsilon)} \leq C n^{-(1+\epsilon)}.
\end{aligned}
\end{equation}

Now we deal with the probability of $(E_n^\epsilon)^c$.
For $n, \tilde{n}>0$, define
\begin{equation*}
\mathrm{F}_{n, \tilde{n}}:=\sum_{i=1}^n \sum_{j=i+1}^{(i+2 \lambda \log n) \wedge n} 1_{\left\{ \min \left(\xi\left(S_i, n\right), \xi\left(S_j, n\right)\right) \geq (1-\epsilon) \lambda \log \tilde{n}, S_i, S_j \notin S_{(i, j)}, S_i \neq S_j\right\}}.
\end{equation*}
Note that for any $k \in\left[n^\rho,(n+1)^\rho\right)$, we have the bound $\mathrm{F}_k \leq \mathrm{F}_{(n+1)^\rho, n^\rho}$.
Consequently,
\begin{equation}\label{eq:4.4}
\bigcup_{n \geq N^\rho}\left(E_n^\epsilon\right)^{c} \subset \bigcup_{n \geq N}\left\{\mathrm{F}_{(n+1)^\rho, n^\rho}>0\right\}.
\end{equation}

It is elementary to check that for non-negative random variables $X, Y$ and any $q>0$,
\begin{equation}\label{eq:4.5}
\{X+Y \geq q\} \subset \bigcup_{r=0}^{[1 / \epsilon]}\{X \geq r \epsilon q,~ Y \geq [1-(r+1) \epsilon] q\}.
\end{equation}

Let $\hat{S}_l:=S_{i-l}-S_i$ and $\widetilde{S}_l:=S_{j+l}-S_j$. For any $1 \leq i \leq (n+1)^\rho, i < j \leq $ $\left(i+\lambda \log (n+1)^\rho\right) \wedge(n+1)^\rho$, by \eqref{eq:4.5}, Proposition \ref{pro:2.5} and the assumption $(1-\epsilon)^2(1+\delta)>1+\delta/2$, we get
\begin{align*}
& \mathbf{P}\left( \min \left(\xi\left(S_i,(n+1)^\rho\right), \xi\left(S_j,(n+1)^\rho\right)\right) \geq(1-\epsilon) \lambda \log n^\rho, S_i, S_j \notin S_{(i, j)}, S_i \neq S_j\right) \\
\leq &
\sum_{z \in \mathbb{Z}\setminus \{0\}}
\mathbf{P}\left(S_j-S_i=z, \sum_{l=0}^i \mathbf{1}_{\left\{\hat{S}_l \in\{0, z\}\right\}}+\sum_{l=0}^{(n+1)^\rho-j} \mathbf{1}_{\left\{\widetilde{S}_l \in\{0,-z\}\right\}} \geq 2(1-\epsilon) \lambda \log n^\rho\right)\\
\leq &
\sum_{z \in \mathbb{Z}\setminus\{0\}}
\mathbf{P}\left(S_j-S_i=z\right) \cdot \sum_{r=0}^{[1 / \epsilon]} \mathbf{P}\left(\sum_{l=0}^i \mathbf{1}_{\left\{\hat{S}_l \in\{0, z\}\right\} } \geq 2 r \epsilon(1-\epsilon) \lambda \log n^\rho\right) \\
&~~ \cdot \mathbf{P}\left(\sum_{l=0}^{(n+1)^\rho-j} \mathbf{1}_{\left\{\widetilde{S}_l \in\{0,-z\}\right\} }\geq 2[1-(r+1) \epsilon](1-\epsilon) \lambda \log n^\rho\right) \\
\leq & \sup_{r\in(0, [1/\epsilon])} ([1/\epsilon]+1)\cdot \sup_{z\in \mathbb{Z}^+} \mathbf{P}\left( \xi(0, (n+1)^\rho)+\xi(z, (n+1)^\rho)\geq 2 r \epsilon(1-\epsilon) \lambda \log n^\rho\right)\\
&~~\cdot  \mathbf{P}\left( \xi(0, (n+1)^\rho)+\xi(z, (n+1)^\rho)\geq 2[1-(r+1) \epsilon](1-\epsilon) \lambda \log n^\rho\right)\\
\leq & \sup_{r\in(0, [1/\epsilon])} ([1 / \epsilon]+1)\cdot  \sup_{z\in \mathbb{Z}^+} \mathbf{P}(\xi(0,\infty)+\xi(z,\infty)>  2 r \epsilon(1-\epsilon) \lambda \log n^\rho)\\
&~~\cdot \mathbf{P}(\xi(0,\infty)+\xi(z,\infty)> 2[1-(r+1) \epsilon](1-\epsilon) \lambda \log n^\rho)\\
\leq & \sup_{r\in(0, [1/\epsilon])} ([1/\epsilon]+1) C \cdot n^{-r \epsilon (1-\epsilon)(\delta +1) \rho }\cdot n^{-[1-(r+1)\epsilon](1-\epsilon)(1+\delta) \rho}\\
= & ([1/\epsilon]+1) C \cdot n^{-(1-\epsilon)^2 (1+\delta) \rho} \leq C \cdot
n^{-\rho (1+ \frac{\delta}{2})}.
\end{align*}
Put $\nabla := (i+2\lambda \log (n+1)^\rho) \wedge (n+1)^\rho$.
Summing over $i$, $j$  and using the assumption that $\frac{\delta \rho }{2}>1+\delta$, we get
\begin{align}\label{eq:4.6}
&\mathbf{P}\left(\mathrm{F}_{(n+1)^\rho, n^\rho}>0\right)\notag\\
= & \mathbf{P}\left( \sum_{i=1}^{(n+1)^\rho} \sum_{j=i+1}^{\nabla} \mathbf{1}_{\left\{ \min \left(\xi\left(S_i, n\right), \xi\left(S_j, n\right)\right) \geq (1-\epsilon) \lambda \log \tilde{n}, S_i, S_j \notin S_{(i, j)}, S_i \neq S_j\right\}} >0 \right)\notag\\
\leq & \sum_{i=1}^{(n+1)^\rho} \sum_{j=i+1}^{\nabla}  \mathbf{P}\left( \min \left(\xi\left(S_i,(n+1)^\rho\right), \xi\left(S_j,(n+1)^\rho\right)\right) \geq (1-\epsilon) \lambda \log n^\rho, S_i, S_j \notin S_{(i, j)}, S_i \neq S_j\right)\notag\\
\leq & (n+1)^\rho \cdot 2 \lambda \log(n+1)^\rho \cdot  C n^{-\rho(1+\frac{\delta}{2})} \notag\\
\leq & C \log n^\rho \cdot n^{-\frac{\delta \rho}{2}}\leq   C n^{-(1+\delta)}.
\end{align}
Combining  \eqref{eq:4.3}, \eqref{eq:4.4} and \eqref{eq:4.6}, we get that, when $N$ is large
\begin{equation*}
  \begin{aligned}
  \mathbf{P} \bigg(\bigcup_{n \geq N}(D_n^\epsilon \cap E_n^\epsilon )^c\bigg)&\leq \sum_{n \geq N}  \mathbf{P} ( (D_n^\epsilon)^c )+ \sum_{n \geq N^{1/\rho}}\mathbf{P} ((E_n^\epsilon )^c) \\
  &\leq  \sum_{n \geq N}    C n^{-(1+\epsilon)} +\sum_{n \geq N^{1/\rho}} n^{-(1+\delta)}=C N^{-(\epsilon \wedge (\delta /\rho))}.
  \end{aligned}
\end{equation*}
\end{proof}

\begin{lem}\label{lem:4.2}
Almost surely, there exists $M=M(\omega)\in \mathbb{Z}^+$ such that for all $m>M$ and $k\ge 2$, $1_{C^k_m}(\omega)=1_{B^k_m}(\omega)$.
\end{lem}

\begin{proof}[\bf Proof]
 Recall  that (cf. \eqref{eq:2.5} and \eqref{eq:2.7}), for any $k\geq 2$,
\begin{align}\label{eq:4.7}
B_m^{k}=A_m^1\cap \cdots \cap A_m^{k},\ \tilde{B}_m^{k}=\tilde{A}_m^2\cap \cdots \cap \tilde{A}_m^{k},\ C_m^{k}=B_m^{k}\cap \tilde{B}_m^{k}.
\end{align}
It follows that
\begin{align}\label{eq:4.8}
C_m^{k}=C_m^{k-1}\cap A_m^{k}\cap \tilde{A}_m^{k}.
\end{align}

By Lemma \ref{lem:4.1}, there exists a null event $\cal{N}$ such that for every $\omega \in \Omega\setminus \cal{N}$, there exists an integer $N'=N'(\omega)\in \mathbb{Z}^+$ satisfying $\omega\in
D_n^\epsilon \cap E_n^\epsilon$ for all $n>N'$.
From now on we work with $\omega\in \Omega\setminus{\cal N}$. If $\xi(n, \omega)=m$, then
\begin{equation}\label{eq:4.9}
\exp \left(\frac{m}{2(1+\epsilon / 2) \lambda}\right) \leq n \leq \exp \left(\frac{m}{(1-\epsilon / 2) \lambda}\right).
\end{equation}

The rest of the proof is divided into two parts.

(i) We first show that for sufficiently large $m$ and all $k\ge 2$, if  $1_{C^k_m}(\omega)=1$, then $1_{\widetilde{A}^{k+1}_m}(\omega)=1$. We prove this by contradiction.

Fix $M^{\prime}=M^{\prime}(\epsilon)>0$ such that for all $m\geq M^{\prime}$ the inequality
\begin{equation}\label{eq:4.10}
(1-\epsilon) \lambda \log \left\{\exp \left(\frac{m+1}{(1-\epsilon / 2) \lambda}\right)+\frac{m}{2}\right\} \leq m
\end{equation}
holds. Set
\begin{align}\label{eq:4.11}
M:=M(\epsilon, \omega):=M^{\prime} \vee (4 \lambda \log N').
\end{align}
Assume that for some $m>M$ and $k \geq 2$, $1_{C_m^k}(\omega)=1$ but $1_{\widetilde{A}_m^{k+1}}(\omega)=0$.  When $1_{C_m^k}(\omega)=1$, $\left\{S_n\right\}_{n \geq 0}$ must hit $L_m^{k}$ at time $T_m^{k}$. When $1_{\widetilde{A}_m^{k+1}}(\omega)=0$, in the time interval $(T_m^k, T_m^k+\frac{m}{2}]$,  the random walk must also visit
one of the sites in $\{ L_m^1, \ldots, L_m^{k-1}\}$, say $L_m^{k'}$.  By \eqref{eq:4.9} we have
\begin{equation}\label{eq:4.12}
\exp \left(\frac{m}{4 \lambda}\right) \leq T_m^1<T_m^k <T_{m+1}^1 \leq \exp \left(\frac{m+1}{(1-\epsilon / 2) \lambda}\right).
\end{equation}
Define $n':=T_m^{k}+\frac{m}{2}$. From \eqref{eq:4.11} we have $m>M \geq 4\lambda \log N' $. Together with \eqref{eq:4.12} this implies
$$n'=T_m^{k}+\frac{m}{2}> \exp\left(\frac{m}{4\lambda}\right)+\frac{m}{2} > N' .$$
Combining \eqref{eq:4.10} with \eqref{eq:4.12} yields
$$m\geq (1-\epsilon) \lambda \log n'  ~~\text{and}~~   2 \lambda \log n' > \frac{m}{2}.$$
Therefore at time $n'$ both $L_m^{k}$ and $L_m^{k'}$ have local time at least $(1-\epsilon)\lambda \log n'$, and the difference between their hitting times is at most $2 \lambda \log n'$.  Hence the pair $\{L_m^{k},L_m^{k'}\}$ is counted by $\mathrm{F}_{n'}$, so $E_{n'}^\epsilon=\{\mathrm{F}_{n'}=0\}$ fails. This contradicts the assumption that
$D_{n'}^\epsilon\cap E_{n'}^\epsilon$
 holds. Consequently such an $m$ cannot exist, which completes the proof by contradiction.

(ii)  Next, we show that for sufficiently large $m$, if $1_{B_m^k}(\omega)=1$, then  $1_{C_m^{k}}(\omega)=1$.
Let $M$ be defined by \eqref{eq:4.11}. We claim that for all $m>M$ and $k\geq 2$, if $1_{B_m^{k}}(\omega)=1$  then $1_{C_m^{k}}(\omega)=1$. By the definition \eqref{eq:2.6}, we know that $1_{C_m^1}(\omega)=1$.  Then, by \eqref{eq:4.7}, \eqref{eq:4.8} and (i), we obtain by induction that $1_{C_m^j}(\omega)=1$ for $j=2, \ldots, k$.

The proof is now complete.
\end{proof}

\begin{proof}[\noindent{\bf Proof of Theorem \ref{thm:1.2}}]   By Lemma \ref{lem:2.4}, it suffices to show that almost surely,
  \begin{equation}\label{eq:4.13}
  \limsup_{ m \to \infty} \frac{ \mathcal{G}_m }{ \log m} = \theta,
  \end{equation}
  where $\theta$ is defined in \eqref{theta-delta} and
\begin{equation*}
  \mathcal{G}_m:=\sup \left\{k: T_m^k<T_{m+1}^1\right\}.
\end{equation*}
For any $\epsilon>0$ satisfying the assumption of Lemma~\ref{lem:4.1}, put
\begin{equation*}
  K_m := \bigcap_{n\ge \exp(m/4\lambda)} (D^\epsilon_n \cap E^\epsilon_n).
\end{equation*}
By Lemma~\ref{lem:4.1} we have $\mathbf{P}(K_m)\to 1$ as $m\to\infty$.
From now on we only consider $m$ large.

(i) Upper Bound. For any $\epsilon>0$, define $$I_m:=[(1+\epsilon) \theta \log m].$$ On $K_m$, if $\left\{\mathcal{G}_m>I_m\right\}=C_m^{I_m+1}$ occurs, then  by Lemma \ref{lem:2.4} and the definition of $E_n^\epsilon$, we have
\begin{equation}\label{eq:4.14}
T_{m+1}^1-T_m^j>T_m^{j+1}-T_m^j \geq \frac{m}{2},
\end{equation}
and thus
\begin{equation*}
S_{T_m^j} \notin S_{\left[T_m^j+1, T_m^j+m / 2\right)}\ \mbox{for any}\ j=1, \ldots, I_m.
\end{equation*}
By applying the strong Markov property successively at the times  $T_m^j, j=1, \ldots, I_m$, Lemma \ref{lem:2.1} and \eqref{eq:2.2}, we obtain
\begin{align*}
&~~~~\mathbf{P}\left(\left\{\mathcal{G}_m>I_m\right\} \cap K_m\right) \\
& \leq \mathbf{P}\bigg(\bigcap_{j=1}^{I_m} \left\{ S_{T_m^j} \notin S_{[T_m^{j}+1, T_m^j+m/2)} \right\} \bigg) = \left[\mathbf{P}\left(S_0 \notin S_{[1, m / 2)}\right)\right]^{I_m} \\
& \leq \left[\mathbf{P}\left(S_0 \notin S_{[1, \infty)}\right)+\mathbf{P}\left(S_0 \in S_{[m / 2, \infty)}\right)\right]^{I_m}
\leq \left( (1-2q)+ \frac{(4pq)^{m/4}}{1-\sqrt{4pq}}\right)^{I_m} \\
& = \exp\left(  (1+\epsilon) \left[-\frac{1}{\log  (1-2q)}\right]\log m \cdot \log\left[(1-2q)+ \frac{(4pq)^{m/4}}{1-\sqrt{4pq}}\right]  \right) \leq C m^{-(1+\epsilon)}.
\end{align*}
Then Lemma \ref{lem:4.1} yields
\begin{equation*}
\mathbf{P}\left(\mathcal{G}_m>I_m\right) \leq \mathbf{P}\left(\left\{\mathcal{G}_m \geq I_m\right\} \cap K_m\right)+\mathbf{P}\left(K_m^{\mathrm{c}}\right) \leq C m^{-(1+\epsilon)}+ C e^{-\frac{m ( \epsilon \wedge  (\delta/\rho))}{4\lambda}}.
\end{equation*}
The upper bounded now follows from the Borel-Cantelli lemma.

(ii) Lower bound.
Denote by $k'$ the index $\in \{1,\ldots,k-2\}$ such that $L_m^{k'}:=\max_{x \in \{1,\cdots, k-2\}} L_m^x$.  Let $n':=T_m^{k-1}-T_m^{k'}$.
Similar to \eqref{eq:4.14}, on $K_m$, we have $n' \ge \tfrac{m}{2}$ when $m$ is large.
By Lemma \ref{lem:2.2},  Stirling's formula and the fact $h=\frac{q}{p}$, we have
\begin{align}\label{eq:4.15}
\mathbf{P}\left(\left\{\sum_{j=1}^{n'}  X_j<0\right\} \cap K_m\right)
 &\leq \mathbf{P}\left(\{S_{n'}<0\}\cap \{n'\geq \frac{m}{2}\}\right)\leq \sum_{n\geq m/2}\mathbf{P}(S_n<0) \nonumber\\
 &=\sum_{n\geq m/2}\sum_{z=1}^n\mathbf{P}(S_n=-z)\leq \sum_{n\geq m/2}\sum_{z=1}^n\binom{n}{[\frac{n-z}{2}]}p^{[\frac{n-z}{2}]}q^{[\frac{n+z}{2}]} \nonumber\\
 &\leq \sum_{n\geq m/2} \sum_{z=1}^n\binom{n}{[\frac{n}{2}]}p^{[\frac{n-z}{2}]}q^{[\frac{n+z}{2}]}\nonumber\\
 &\leq C\sum_{n\geq m/2}\sum_{z=1}^n2^n\sqrt{\frac{2}{\pi n}}(pq)^{\frac{n}{2}}\left(\frac{q}{p}\right)^{\frac{z}{2}} \nonumber\\
 &=C\sum_{n\geq m/2}(4pq)^{\frac{n}{2}}\sqrt{\frac{2}{\pi n}}\frac{\sqrt{h}(1-\sqrt{h}^n)}{1-\sqrt{h}} \nonumber\\
 &\leq C\sqrt{\frac{2}{\pi}}\frac{\sqrt{h}}{1-\sqrt{h}}\sum_{n\geq m/2}(4pq)^{\frac{n}{2}}\frac{1}{\sqrt{n}}
 \leq C\frac{(4pq)^{m/4}}{\sqrt{m}}.
\end{align}
Define
$\alpha_1:=-\frac{1}{4}\log(4pq)$,
$\alpha_2:=\frac{\epsilon \wedge \delta/\rho}{4\lambda}$,
$\eta:=\frac{1}{2}\min\{\alpha_1,\alpha_2\}>0$.
Applying  Lemma~\ref{lem:4.1} again, we get that for all large enough $m$,
\begin{equation}\label{eq:4.16}
\begin{aligned}
 \mathbf{P}\left(\sum_{j=1}^{n'}  X_j<0 \right) \leq&  \mathbf{P}\left(\left\{\sum_{j=1}^{n'}  X_j<0\right\} \cap K_m\right)+\mathbf{P}(K_m^c)\\
 \leq &C\frac{(4pq)^{m/4}}{\sqrt{m}} +C e^{-\frac{m ( \epsilon  \wedge \delta/\rho)}{4\lambda}}\\
 =&   C e^{-\alpha_1 m} m^{-1/2} +C e^{-\alpha_2 m}\leq C e^{-\eta m}.
\end{aligned}
\end{equation}

For $A\subset \mathbb{Z}$ and $n\in \mathbb{N}$, define
\begin{equation*}
  \mathrm{H}_A(n):= \inf \{k \geq n: S_k \in A\}-n.
\end{equation*}
If $A=\{z\} \subset \mathbb{Z}$, we write $\mathrm{H}_A(n)$ as $\mathrm{H}_z(n)$.
 Recalling the definition of $A_m^k$, for $k>1$ we have
\begin{equation}\label{eq:4.17}
\begin{aligned}
    A_m^k \supset& \left\{\mathrm{H}_{L_m^{k-1}}\left(T_m^{k-1}\right)=\infty\right\} \cap  \left\{\mathrm{H}_{\left\{L_m^1, \ldots, L_m^{k-1}\right\}}\left(T_m^{k-1}+\frac{m}{2}\right)=\infty\right\}.
\end{aligned}
\end{equation}
Combining \eqref{eq:4.16}, \eqref{eq:4.17} and Lemma~\ref{lem:2.1}, and noting that $n'$ is $\mathcal F_{T_m^{k-1}}$-measurable, an application of the strong Markov property at the stopping time $T_m^{k-1}$ yields, almost surely,
\begin{align}\label{eq:4.18}
\mathbf{P}\left(A_m^k \mid \mathcal{F}_m^{k-1}\right) \geq& \mathbf{P}\left( \left\{\mathrm{H}_{L_m^{k-1}}\left(T_m^{k-1}\right)=\infty\right\} \bigcap  \left\{\mathrm{H}_{\left\{L_m^1, \ldots, L_m^{k-1}\right\}}\left(T_m^{k-1}+\frac{m}{2}\right)=\infty\right\} ~\big{|}~ \mathcal{F}_m^{k-1}\right) \nonumber\\
=&\mathbf{P}\left(\left\{\mathrm{H}_{L_m^{k-1}}\left(T_m^{k-1}\right)=\infty\right\}  \big|~ \mathcal{F}_m^{k-1}\right)\nonumber\\
&~-\mathbf{P}\left(\left\{ \mathrm{H}_{L_m^{k-1}}(T_m^{k-1})=\infty\right\}\bigcap \left\{\mathrm{H}_{\left\{L_m^1, \ldots, L_m^{k-1}\right\}}\left(T_m^{k-1}+\frac{m}{2} \right)< \infty\right\}\big|~ \mathcal{F}_m^{k-1}\right) \nonumber\\
\geq& \gamma -\mathbf{P}\left( \left\{L_m^{k'}>L_m^{k-1}\right\} \bigcap \left\{ \mathrm{H}_{L_m^{k'}}\left(T_m^{k-1}+\frac{m}{2}\right)< \infty\right\} \big|~ \mathcal{F}_m^{k-1} \right)\nonumber\\
=& \gamma -\mathbf{P} \left( \left\{ \sum_{j=T_m^{k'}+1}^{T_m^{k-1}} X_j <0   \right\} \bigcap \left\{ \mathrm{H}_{L_m^{k'}} \left(T_m^{k-1}+\frac{m}{2}\right) < \infty\right\} \bigg|~ \mathcal{F}_m^{k-1} \right) \nonumber\\
=& \gamma -\mathbf{P}\left( \left\{ \sum_{j=T_m^{k'}+1}^{T_m^{k'}+n'} X_j <0   \right\} \bigcap \left\{ \mathrm{H}_{L_m^{k'}}\left(T_m^{k-1}+\frac{m}{2}\right)< \infty\right\} \bigg|~ \mathcal{F}_m^{k-1} \right) \nonumber\\
\geq& \gamma -\mathbf{P}\left(  \sum_{j=T_m^{k-1}+1}^{T_m^{k-1}+n'} X_j <0 ~  \bigg|~ \mathcal{F}_m^{k-1} \right) \nonumber\\
=& \gamma -\sum_{n\geq 1}\mathbf{P}\left(  \sum_{j=T_m^{k-1}+1}^{T_m^{k-1}+n} X_j <0,~n'=n  ~ \bigg|~ \mathcal{F}_m^{k-1}\right) \nonumber\\
=& \gamma -\sum_{n\geq 1}\mathbf{1}_{\{n'=n\}}\cdot \mathbf{P}_{S_{T_m^{k-1}}}\left(  \sum_{j=1}^{n} X_j <0 \right) \nonumber\\
=& \gamma - \mathbf{P}_{S_{T_m^{k-1}}}\left( \sum_{j=1}^{n'} X_j <0    \right) \nonumber\\
\geq& 1-2q- Ce^{-\eta  m}.
\end{align}

Note that $A_m^1$ is the sure event by definition.
Note all that for any $m, k \geq 1$, both $A_m^k, B_m^k$ are in $\mathcal{F}_m^k$.  We define
\begin{equation}\label{eq:4.19}
  J_m:=[(1-\epsilon) \theta \log m].
\end{equation}
By \eqref{eq:4.18} and the definition of $\theta$ in \eqref{theta-delta}, we have
\begin{align*}
&\mathbf{P}\left(B_m^{J_m} \mid \mathcal{F}_m^1\right)= \mathbf{P}(A_m^1 \cap \cdots \cap A_m^{J_m}\mid \mathcal{F}_m^1) \\
\geq&\prod_{j=2}^{J_m} \left(1-2q - Ce^{-\eta m}\right)\\
= &\exp\left[  \sum_{j=2}^{J_m}\log\left\{ \left(1-2q\right)\left(1-\frac{Ce^{-\eta m}}{1-2q}\right)\right\} \right]\\
  = & \exp[(J_m-1) \log (1-2q)]\cdot \exp\left[(J_m-1)\log \left(1-\frac{Ce^{-\eta m}}{1-2q}\right) \right]\\
\geq & C\exp[(1-\epsilon)\theta \log m \cdot\log (1-2q)]\cdot \exp\left[(1-\epsilon)\theta \log m \cdot \log \left(1-\frac{C}{1-2q} e^{-\eta m}\right) \right]\\
= & C m^{-(1-\epsilon)}\cdot m^{-(1-\epsilon) \cdot \frac{\log\left(1-\frac{C}{1-2q} e^{- \eta m}\right)}{\log(1-2q)}}\geq Cm^{-(1-\epsilon)}.
\end{align*}
Hence
$$
\sum_{m=1}^{\infty}\mathbf{P}\left(B_m^{J_m}\mid\mathcal{F}_m^1\right)\geq \infty.
$$
Note that $B_m^k \in \mathcal{F}_{m+1}^1$.
Therefore, by the generalized second Borel-Cantelli lemma
(see e.g. \cite{D2019}, Theorem~4.3.4),
$B_m^{J_m}$ occurs infinitely often with probability one. By Lemma \ref{lem:4.2}, it follows that $C_m^{J_m}$ also occurs infinitely often almost surely. In other words,
$$
\mathbf{P}(\mathcal{G}_m \leq J_m ~\text{i.o.})=1.
$$
This completes the proof of the lower bound.
\end{proof}

{ \noindent {\bf\large Acknowledgments}\ \ We would like to thank Chenxu Hao and Yushu Zheng for helpful discussions. This work was supported by the National Natural Science Foundation of China (Grant Nos. 12171335, 12471139) and the Simons Foundation (\#429343).

%

\end{document}